\documentclass[11pt, a4paper]{amsart}

\usepackage{amsfonts,amsmath,amssymb, amscd}

\newtheorem{theorem}{Theorem}[section]
\newtheorem{lemma}[theorem]{Lemma}
\newtheorem{definition}[theorem]{Definition}
\newtheorem{prp}[theorem]{Proposition}
\newtheorem{crl}[theorem]{Corollary}

\theoremstyle{definition}
\newtheorem{rem}[theorem]{Remark}

\newtheorem{exa}[theorem]{Example}
\newtheorem{exas}[theorem]{Examples}

\newcommand\eps{\varepsilon}
\newcommand\Frac{\operatorname{Frac}}
\newcommand\Alg{\operatorname{Alg}}

\newcommand\BB{\mathcal B}
\newcommand\UU{\mathcal U}
\newcommand\ZZ{\mathcal{Z}}
\newcommand\FF{\mathbb{F}}

\newcommand\ZZZ{\mathbb Z}

\DeclareMathOperator{\Ker}{Ker}
\DeclareMathOperator{\id}{id}

\numberwithin{equation}{section}

\title[Hopf algebras and polynomial identities]
{Hopf algebras and
polynomial identities}

\author{Christian Kassel}
\address{Christian Kassel: 
Institut de Recherche Math\'e\-ma\-tique Avanc\'ee,
CNRS \& Universit\'e de Strasbourg,
7 rue Ren\'{e} Descartes, 67084 Strasbourg, France}
\email{kassel@math.unistra.fr}
\urladdr{www-irma.u-strasbg.fr/\raise-2pt\hbox{\~{}}kassel/}

\begin{document}

\begin{abstract} 
This is a survey of results obtained jointly with E.~Aljadeff and published in~\cite{AK}.
We explain how to set up a theory of polynomial identities for comodule algebras
over a Hopf algebra, and concentrate on the universal comodule algebra
constructed from the identities satisfied by a given comodule algebra.
All concepts are illustrated with various examples.
\end{abstract}

\maketitle

\noindent
{\sc Key Words:}
Polynomial identity, Hopf algebra, comodule, localization

\medskip
\noindent
{\sc Mathematics Subject Classification (2010):}
16R50, %Other kinds of identities
16T05, %Hopf algebras and their applications
16T15, %Coalgebras and comodules
16T20, %Ring-theoretic aspects of quantum groups
16S40, %Smash products of general Hopf actions%
16S85 %Localizations%

\hspace{3cm}

\section*{Introduction}

As has been stressed many times (see, e.g.,~\cite{Sch}), 
Hopf Galois extensions can be viewed as non-commutative analogues of 
principal fiber bundles (also known as $G$-torsors), where 
the role of the structural group is played by a Hopf algebra. 
Such extensions abound in the world of quantum groups and of non-commutative geometry.
The problem of constructing systematically all Hopf Galois extensions of a given algebra 
for a given Hopf algebra and of classifying them up to isomorphism
has been addressed in a number of papers,
such as \cite{Bi, DT, K, KS, Ma1, Ma11, Ma2, Sb1} to quote but a few.

A new approach to the classification problem of Hopf Galois extensions 
was recently advanced by Eli Aljadeff and the present author in~\cite{AK};
this approach uses classical techniques from non-commutative algebra
such as \emph{polynomial identities}
(such techniques had previously been used in~\cite{AHN} for group-graded algebras).
In~\cite{AK} we developed a theory of identities for
any comodule algebra over a given Hopf algebra~$H$, hence for any Hopf Galois extension. 
As a result, out of the identities for an $H$-comodule algebra~$A$, 
we obtained a \emph{universal $H$-comodule algebra}~$\UU_H(A)$.
It turns out that if $A$ is a cleft $H$-Galois object (i.e., 
a comodule algebra obtained from~$H$ by twisting its product 
with the help of a two-cocycle) with trivial center, 
then a suitable central localization of~$\UU_H(A)$ is an $H$-Galois extension of its center.
We thus obtain a ``non-commutative principal fiber bundle" 
whose base space is the spectrum of some localization of the center of~$\UU_H(A)$. 

This survey is organized as follows.
After a preliminary section on comodule algebras, 
we define the concept of an $H$-identity for such algebras in~\S~\ref{sec-ident}.
We illustrate this concept with a few examples 
and we attach a universal $H$-comodule algebra~$\UU_H(A)$
to each $H$-comodule algebra~$A$.

In~\S~\ref{sec-detecting} turning to the special case where $A = {}^{\alpha} H$ is a twisted comodule algebra,
we exhibit a universal comodule algebra map that allows us to detect the $H$-iden\-titi\-es for~$A$. 

In~\S~\ref{sec-local} we construct a commutative domain~$\BB_H^{\alpha}$ and
we state that under some natural extra condition,
$\BB_H^{\alpha}$ is the center of a suitable central localization of~$\UU_H(A)$; 
moreover after localization, $\UU_H(A)$ becomes a free module over its center.

Lastly in~\S~\ref{Sweedler}, 
we illustrate all previous constructions with the help of the four-dimensional Sweedler algebra,
thus giving complete answers in this simple, but non-trivial example.
We end the paper with an open question on Taft algebras.

The material of the present text is mainly taken from~\cite{AK}, 
for which it provides an easy access.
The reader is advised to complement it with~\cite{Ka1, KM}.

\section{Hopf algebras and coactions}\label{prelim}

\subsection{Standing assumption}

We fix a field~$k$ over which all
our constructions are defined. 
In particular, all linear maps are supposed to be $k$-linear
and unadorned tensor products mean tensor products over~$k$.
Throughout the survey we assume that the ground field~$k$ is \emph{infinite}.

By algebra we always mean an associative unital $k$-algebra.
We suppose the reader familiar with the language of Hopf algebra, 
as expounded for instance in~\cite{Sw}.
As is customary, we denote the coproduct of a Hopf algebra by~$\Delta$, its counit by~$\eps$,
and its antipode by~$S$.
We also make use of a Heyneman-Sweedler-type notation 
for the image 
\[
\Delta(x) = x_1 \otimes x_2
\]
of an element~$x$ of a Hopf algebra~$H$
under the coproduct, and we write
\[
\Delta^{(2)}(x) = x_1 \otimes x_2 \otimes x_3
\]
for the iterated coproduct 
$\Delta^{(2)} = (\Delta \otimes \id_H) \circ \Delta = (\id_H \otimes \Delta) \circ \Delta$,
and so~on.

\subsection{Comodule algebras}

Let $H$ be a Hopf algebra.
Recall that an $H$-\emph{comodule algebra} is an 
algebra $A$ equipped with a
right $H$-comodule structure whose (coassociative, counital) \emph{coaction}
\[
\delta : A \to A \otimes H
\] 
is an algebra map.
The subalgebra~$A^H$ of \emph{coin\-var\-iants} of an $H$-comodule algebra~$A$
is defined by
$$A^H = \{ a \in A \, | \, \delta(a)  = a \otimes 1\} \, .$$

Given two $H$-comodule algebras $A$ and $A'$ with respective coactions
$\delta$ and~$\delta'$, an algebra map  $f : A \to A'$ 
is an \emph{$H$-comodule algebra map} if
\[
\delta' \circ f = (f\otimes \id_H) \circ \delta \, .
\]
We denote by~$\Alg^H$ the category whose objects are $H$-comodule algebras
and arrows are $H$-comodule algebra maps.

Let us give a few examples of comodule algebras.

\begin{exa}\label{trivial-algebra}
If $H = k$, then an $H$-comodule algebra is nothing but an
ordinary (associative, unital) algebra.
\end{exa}

\begin{exa}\label{graded-algebra}
The algebra $H = k[G]$ of a \emph{group}~$G$
is a Hopf algebra with coproduct, counit, and antipode given for all $g\in G$ by
\[
\Delta(g) = g\otimes g \, , \quad\eps(g) = 1\, , \quad S(g) = g^{-1}\, .
\]
It is well-known (see~\cite[Lemma~4.8]{BM}) that 
an $H$-comodule algebra~$A$ is the same as a $G$-\emph{graded algebra}
\[
A = \bigoplus_{g\in G}\, A_g \, ,\qquad A_g \, A_h \subset A_{gh} \, .
\]
The coaction $\delta : A \to A \otimes H$ is given by
$\delta(a) = a\otimes g$ for all $a\in A_g$ and~$g\in G$.
We have $A^H = A_e$, where $e$ is the neutral element of~$G$.
\end{exa}

\begin{exa}\label{G-algebra}
Let $G$ be a \emph{finite group} and $H = k^G$ be 
the algebra of $k$-valued functions on a finite group~$G$.
This algebra can be equipped with a Hopf algebra structure 
that is dual to the Hopf algebra~$k[G]$ above. 
An $H$-comodule algebra~$A$ is the same as a $G$-\emph{algebra},
i.e., an algebra equipped with a left action of~$G$ on~$A$ by group automorphisms.

If we denote the action of~$g\in G$ on $a\in A$ by ${}^g a$, then
the coaction $\delta : A \to A \otimes H$ is given by
\[
\delta(a) = \sum_{g\in G}\, {}^g a \otimes e_g\, ,
\]
where $\{e_g\}_{g\in G}$ is the basis of~$H$ consisting of
the functions $e_g$ defined by $e_g(h) = 1$ if $h=g$, and~$0$ otherwise.

The subalgebra of coinvariants of~$A$ coincides with the subalgebra of
$G$-invariant elements: $A^H = A^G$.
\end{exa}

\pagebreak[3]

\begin{exa}\label{free-comod-algebra}
Any Hopf algebra~$H$ is an $H$-comodule algebra
whose coaction coincides with the coproduct of~$H$:
\[
\delta = \Delta : H \to H\otimes H \, .
\]
In this case the coinvariants of~$H$ are exactly the
scalar multiples of the unit of~$H$; in other words, $H^H = k\, 1$.
\end{exa}

\section{Identities}\label{sec-ident}

\subsection{Polynomial identities}

Let $A$ be an algebra. A \emph{polynomial identity} for an algebra~$A$
is a polynomial $P(X,Y,Z,\ldots)$ in a finite number of non-commutative variables $X,Y,Z,\ldots$
such that 
\[
P(x,y,z,\ldots) = 0
\]
for all $x,y,z,\ldots \in A$.

\begin{exas}
(a) The polynomial $XY - YX$ is a polynomial identity for any \emph{commutative algebra}.

(b) If $A = M_2(k)$ is the \emph{algebra of $2 \times 2$-matrices} with entries in~$k$, then
\[
(XY - YX)^2 Z -  Z (XY - YX)^2
\]
is a polynomial identity for~$A$. (Use the Cayley-Hamilton theorem to check this.)
\end{exas}

The concept of a polynomial identity first emerged in the 1920's
in an article~\cite{De} on the foundation of projective geometry by Max Dehn, the topologist. 
The above polynomial identity for the algebra of $2 \times 2$-matrices appeared in~1937 in~\cite{Wa}.
Today there is an abundant literature on polynomial identities; see for instance~\cite{Fo, Ro}.

For algebras graded by a group~$G$
there exists the concept of a graded polynomial identity (see~\cite{AHN, BZ}).
In this case we need to take a family of non-commutative variables $X_g,Y_g,Z_g,\ldots$
for each element $g \in G$. Given a $G$-graded algebra
$A = \bigoplus_{g\in G}\, A_g $,
a \emph{graded polynomial identity} is a polynomial~$P$ in these indexed variables  
such that $P$ vanishes upon any substitution of each variable $X_g$
appearing in~$P$ by an element of the $g$-component~$A_g$. 

In general, we should keep in mind that in order to define polynomial identities for a class of algebras, 
we need to single out 

\begin{itemize}
\item[(i)] a suitable algebra of non-commutative polynomials and

\item[(ii)] a suitable notion of specialization for these polynomials.
\end{itemize}

\pagebreak[3]

The algebras of interest to us in this survey are comodule algebras over a Hopf algebra~$H$.
The non-commutative variables we wish to use will be indexed by the elements of some linear basis of~$H$.
Since in general a Hopf algebra does not have a natural basis, we find it preferable to use a more
canonical construction, namely the tensor algebra over~$H$, 
and to resort to a given basis only when we need to perform computations.

\pagebreak[3]

\subsection{Definition and examples of $H$-identities}\label{sec-identities}

Let $H$ be a Hopf algebra. We pick a copy~$X_H$ of the underlying vector space of~$H$
and we denote the identity map from $H$ to~$X_H$ by $x\mapsto X_x$  for all $x\in H$.

Consider the \emph{tensor algebra}~$T(X_H)$ of the vector space~$X_H$
over the ground field~$k$:
\[
T(X_H) = \bigoplus_{r\geq 0}\, T^r(X_H) \, ,
\]
where $T^r(X_H) = X_H^{\otimes r}$ is the tensor product of $r$ copies of~$X_H$ over~$k$,
with the convention $T^0(X_H) = k$.
If $\{x_i\}_{i\in I}$ is some linear basis of~$H$, 
then $T(X_H)$ is isomorphic to the algebra of non-commutative polynomials
in the indeterminates~$X_{x_i}$ ($i\in I$).

Beware that the product $X_x X_y$ of symbols in the tensor algebra is different
from the symbol~$X_{xy}$ attached to the product of~$x$ and~$y$ in~$H$;
the former is of degree~$2$ while the latter is of degree~$1$.

The algebra $T(X_H)$ is an $H$-\emph{comodule algebra} 
equipped with the coaction 
\begin{equation*}\label{TXcoaction2}
\delta : T(X_H) \to T(X_H) \otimes H \; ; \;\; X_x \mapsto X_{x_1} \otimes x_2 \, .
\end{equation*}

Note that $T(X_H)$ is \emph{graded} with all generators~$X_x$ in degree~$1$.
The coaction preserves the grading, where $T(X_H) \otimes H$ is graded by
\[
(T(X_H) \otimes H)_r = T^r(X_H) \otimes H
\]
for all $r\geq 0$.

We now give the main definition of this section.

\begin{definition}\label{def-invariant}
Let $A$ be an $H$-comodule algebra.
An element $P \in T(X_H)$ is an $H$-identity for~$A$ if
$\mu(P) = 0$ for all $H$-comodule algebra maps 
\[
\mu : T(X_H) \to A \, .
\]
\end{definition}

\pagebreak[3]

To convey the feeling of what an $H$-identity is, let us give some simple examples.

\begin{exa}\label{trivial-identity}
Let $H = k$ be the one-dimension Hopf algebra as in Example~\ref{trivial-algebra}.
An $H$-comodule algebra~$A$ is then the same as an algebra.
In this case, $T(X_H)$ coincides with the polynomial algebra~$k[X_1]$
and an $H$-comodule algebra map is nothing but an algebra map.
Therefore, an element $P(X_1) \in T(X_H) = k[X_1]$ is an $H$-identity for~$A$ if and only
if all $P(a) = 0$ for all $a \in A$.
Since $k$ is assumed to be infinite, it follows that there are no non-zero $H$-identities for~$A$.
\end{exa}

\pagebreak[3]

\begin{exa}\label{graded-identity}
Let $H = k[G]$ be a group Hopf algebra as in Example~\ref{graded-algebra}.
We know that an $H$-comodule algebra is a $G$-graded algebra
$A = \bigoplus_{g\in G}\, A_g$.
Since $\{g\}_{g\in G}$ is a basis of~$H$,
the tensor algebra~$T(X_H)$ is the algebra of non-commutative polynomials
in the indeterminates~$X_g$ ($g\in G$).

It is easy to check that an algebra map $\mu : T(X_H) \to A$
is an $H$-comodule algebra map if and only if 
$\mu(X_g) \in A_g$ for all $g\in G$.
This remark allows us to produce the following examples of~$H$-identities.

\begin{itemize}
\item[(a)]
Suppose that $A$ is \emph{trivially graded}, i.e., $A_g = 0$ for all $g \neq e$.
Then any non-commutative polynomial in the indeterminates $X_g$ with $g\neq e$
is killed by any $H$-comodule algebra map $\mu : T(X_H) \to A$.
Therefore, such a polynomial is an $H$-identity for~$A$.

\item[(b)]
Suppose that the trivial component~$A_e$ is \emph{central} in~$A$.
We claim that 
\[
X_g X_{g^{-1}} X_h - X_hX_g X_{g^{-1}}
\]
is an $H$-identity for~$A$ for all $g,h \in G$.
Indeed, for any $H$-comodule algebra map $\mu : T(X_H) \to A$,
we have 
\[
\mu(X_g) \in A_g
\quad\text{and}\quad 
\mu(X_{g^{-1}}) \in A_{g^{-1}} \, ;
\]
therefore, $\mu(X_g X_{g^{-1}}) = \mu(X_g) \, \mu(X_{g^{-1}})$ belongs to~$A_e$, hence
commutes with~$\mu(X_h)$.
One shows in a similar fashion that if $g$ is an element of~$G$ of finite order~$N$,
then for all $h \in G$,
\[
X_g^N  X_h - X_hX_g^N
\]
is an $H$-identity for~$A$.
\end{itemize}

\end{exa}

\begin{exa}\label{coinvariant-identity}
Let $H$ be an arbitrary Hopf algebra, and let $A$ be an $H$-comodule algebra such that
the subalgebra~$A^H$ of coinvariants is central in~$A$
(the twisted comodule algebras of \S~\ref{twisted} satisfy the latter condition).

For $x,y \in H$ consider the following elements of~$T(X_H)$:
\begin{equation*}
P_x = X_{x_1} \, X_{S(x_2)} 
\quad\text{and}\quad
Q_{x,y} = X_{x_1} \, X_{y_1} \, X_{S(x_2y_2)} \, .
\end{equation*}
Then for all $x,y,z \in H$,
\begin{equation*}
P_x \, X_z - X_z \, P_x
\quad\text{and}\quad
Q_{x,y} \, X_z - X_z \, Q_{x,y}
\end{equation*}
are $H$-identities for~$A$.
Indeed, $P_x$ and $Q_{x,y}$ are coinvariant elements of~$T(X_H)$;
see~\cite[Lem\-ma~2.1]{AK}.
It follows that for any $H$-comodule algebra map $\mu : T(X_H) \to A$,
the elements $\mu(P_x)$ and~$\mu(Q_{x,y})$ are coinvariant, hence central, in~$A$.
\end{exa}

More sophisticated examples of $H$-identities will be given in~\S~\ref{Sweedler}.

\subsection{The ideal of $H$-identities}

Let $H$ be a Hopf algebra and $A$ an $H$-comodule algebra.
Denote the set of all $H$-identities for~$A$ by~$I_H(A)$.
By definition,
\begin{equation*}\label{I(A)-def}
I_H(A) = \bigcap_{\mu \, \in \, \Alg^{H}(T(X_H),A)}\, \Ker \mu \, .
\end{equation*}

A proof of the following assertions can be found in~\cite[Prop.~2.2]{AK}.

\begin{prp}
The set $I_H(A)$ has the following properties:

(a) it is a graded ideal of~$T(X_H)$, i.e.,
\[
I_H(A) \, T(X_H) \subset I_H(A) \supset T(X_H) \, I_H(A)
\]
and
\[
I_H(A) = \bigoplus_{r\geq 0} \,  \left( I_H(A) \, \bigcap \, T^r(X_H) \right)  ;
\]

(b) it is a right $H$-coideal of~$T(X_H)$, i.e.,
\[
\delta\bigl(I_H(A)\bigr) \subset I_H(A) \otimes H \, .
\]
\end{prp}

Note that for any $H$-comodule algebra map $\mu : T(X_H) \to A$,
we have $\mu(1) = 1$; therefore,
the degree~$0$ component of~$I_H(A)$ is always trivial:
\[
I_H(A) \, \bigcap \, T^0(X_H) = 0 \, .
\]
If, in addition, there exists an injective $H$-comodule map $H \to A$, then
the degree~$1$ component of~$I_H(A)$ is also trivial:
\[
I_H(A) \, \bigcap \, T^1(X_H) = 0 \, .
\]

\begin{rem}
Right from the beginning we required the ground field~$k$ to be infinite.
This assumption is used for instance to establish that $I_H(A)$ is a graded ideal of~$T(X_H)$.
Let us give a proof of this fact in order to show how the assumption is used.
Indeed, expand $P \in I_H(A)$ as
\[
P = \sum_{r\geq 0}\, P_r 
\]
with $P_r \in T^r(X_H)$ for all $r\geq 0$. 
To prove that $I_H(A)$ is a graded ideal, 
it suffices to check that each~$P_r$ is in~$I_H(A)$.
Given a scalar $\lambda \in k$, consider the algebra endomorphism $\lambda_*$ of~$T(X_H)$ 
defined by $\lambda(X_x) = \lambda X_x$ for all $x\in H$;
clearly, $\lambda_*$ is an $H$-comodule map.
If $\mu : T(X_H) \to A$ is an $H$-comodule algebra map,
then so is $\mu \circ \lambda_*$. Since $P \in I_H(A)$, we have
\[
\sum_{r\geq 0}\, \lambda^r \mu(P_r) = (\mu \circ \lambda_*)(P) = 0\, .
\]
The $A$-valued polynomial $\sum_{r\geq 0}\, \lambda^r \mu(P_r)$ takes zero values for all $\lambda \in k$.
By the assumption on~$k$, this implies that its coefficients are all zero,
i.e., $\mu(P_r) = 0$ for all $r\geq 0$. 
Since this holds for all $\mu \in \Alg^{H}(T(X_H),A)$,
we obtain $P_r \in I_H(A)$ for all $r\geq 0$.

If the ground field is \emph{finite}, then Definition~\ref{def-invariant} still makes sense, 
but the ideal~$I_H(A)$ may no longer be graded.
Indeed, let $k$ be the finite field~$\FF_p$ and $H= k$. 
Then for $q = p^N$, the finite field~$\FF_q$ is an $H$-comodule algebra.
In view of Example~\ref{trivial-identity}, the polynomial $X_1^q - X_1$
is an $H$-identity for~$\FF_q$, 
but clearly the homogeneous summands in this polynomial, namely $X_1^q$ and~$X_1$,
are not $H$-identities.
\end{rem}

\subsection{The universal $H$-comodule algebra}\label{univ-comod-alg}

Let $A$ be an $H$-comodule algebra and $I_H(A)$ the ideal of $H$-identities for~$A$
defined above. Since $I_H(A)$ is a graded ideal of~$T(X_H)$,
we may consider the quotient algebra
\begin{equation*}\label{def-U}
\UU_H(A) = T(X_H)/I_H(A) \, .
\end{equation*}
The grading on~$T(X_H)$ induces a grading on~$\UU_H(A)$.
As $I_H(A)$ is a right $H$-coideal of~$T(X_H)$,
the quotient algebra~$\UU_H(A)$ carries an $H$-comodule algebra structure inherited from~$T(X_H)$.

By definition of~$\UU_H(A)$, all $H$-identities for~$A$ vanish in~$\UU_H(A)$.
For this reason we call~$\UU_H(A)$ the 
\emph{universal $H$-comodule algebra attached to}~$A$.

The algebra~$\UU_H(A)$ has two interesting subalgebras:

\begin{itemize}
\item[(i)] 
The subalgebra $\UU_H(A)^H$ of \emph{coinvariants} of~$\UU_H(A)$.

\item[(ii)] 
The \emph{center}~$\ZZ_H(A)$ of~$\UU_H(A)$.
\end{itemize}

We now raise the following question.
Suppose that the comodule algebra~$A$
is free as a module over the subalgebra of coinvariants~$A^H$ (or over its center);
is $\UU_H(A)$, or rather some suitable central localization of it, then free 
as a module over some localization of~$\UU_H(A)^H$ (or of~$\ZZ_H(A)$)?
An answer to this question will be given below (see Theorem~\ref{answer})
for a special class of comodule algebras, which we introduce in the next section.

\section{Detecting $H$-identities}\label{sec-detecting}

Fix a Hopf algebra~$H$. We now define a special class of $H$-comodule algebras
for which we can detect all~$H$-identities.

\subsection{Twisted comodule algebras}\label{twisted}

Recall that a \emph{two-cocycle}~$\alpha$ on~$H$ is 
a bilinear form $\alpha : H \times H \to k$ such that
\begin{equation*}\label{cocycle}
\alpha(x_1,y_1)\, \alpha(x_2 y_2, z)
= \alpha(y_1, z_1)\, \alpha(x, y_2 z_2)
\end{equation*}
for all $x,y,z \in H$.
We assume that $\alpha$ is convolution-invertible and write~$\alpha^{-1}$ for its inverse.
For simplicity, we also assume that $\alpha$ is normalized, i.e.,
\begin{equation*}\label{normalized}
\alpha(x,1)  = \alpha(1,x) = \varepsilon(x)
\end{equation*}
for all $x\in H$. 

Any Hopf algebra possesses at least one normalized convolution-invertible two-cocycle, 
namely the \emph{trivial} two-co\-cycle~$\alpha_0$,
which is defined by
\begin{equation*}
\alpha_0(x,y) = \eps(x)\, \eps(y) 
\end{equation*}
for all $x,y \in H$.

Let $u_H$ be a copy of the underlying vector space of~$H$.
Denote the identity map from $H$ to~$u_H$ by $x \mapsto u_x$ ($x\in H$).
We define the \emph{twisted algebra} ${}^{\alpha} H$ as the vector space~$u_H$ equipped 
with the associative product given by
\begin{equation*}\label{twisted-multiplication}
u_x \,   u_y = \alpha(x_1, y_1) \, u_{x_2 y_2}
\end{equation*}
for all $x$, $y \in H$.
This product is associative because of the above cocycle condition;
the two-cocycle~$\alpha$ being normalized,
$u_1$ is the unit of~${}^{\alpha} H$.

The algebra ${}^{\alpha} H$ is an $H$-comodule algebra
with coaction 
$\delta \colon {}^{\alpha} H \to {}^{\alpha} H \otimes H$
given for all $x\in H$ by
\[
\delta (u_x) = u_{x_1} \otimes x_2 \, .
\] 
It is easy to check that the subalgebra of coinvariants of~${}^{\alpha} H$
coincides with~$k \, u_1$, which lies in the center of~${}^{\alpha} H$.

Note that if $\alpha = \alpha_0$ is the trivial two-cocycle, 
then ${}^{\alpha}H = H$ is the $H$-comodule algebra of Example~\ref{free-comod-algebra}.

The twisted comodule algebras of the form~${}^{\alpha} H$
coincide with the so-called \emph{cleft $H$-Galois objects}; see~\cite[Prop.~7.2.3]{M2}.
It is therefore an important class of comodule algebras.
We next show how we can detect $H$-identities for such comodule algebras.

\subsection{The universal comodule algebra map}\label{univ-map}

We pick a third copy~$t_H$ of the underlying vector space of~$H$
and denote the identity map from $H$ to~$t_H$ by $x\mapsto t_x$ ($x\in H$).
Let $S(t_H)$ be the \emph{symmetric algebra} over the vector space~$t_H$.
If $\{x_i\}_{i\in I}$ is a linear basis of~$H$, then $S(t_H)$ is
isomorphic to the (commutative) algebra of polynomials in the indeterminates $t_{x_i}$ ($i\in I$).

We consider the algebra~$S(t_H) \otimes {}^{\alpha}H$. 
As a $k$-algebra, it is generated by the symbols $t_z u_x$ ($x,z \in H)$
(we drop the tensor product sign~$\otimes$ between the $t$-symbols and the $u$-symbols). 

The algebra~$S(t_H) \otimes {}^{\alpha}H$ is an $H$-comodule algebra 
whose $S(t_H)$-linear coaction extends the coaction of~${}^{\alpha}H$:
\begin{equation*}
\delta(t_z u_x) = t_z u_{x_1} \otimes x_2 \, .
\end{equation*}

Define an algebra map $\mu_{\alpha} : T(X_H) \to S(t_H) \otimes {}^{\alpha}H$ by
\begin{equation*}
\mu_{\alpha}(X_x) = t_{x_1} \, u_{x_2} 
\end{equation*}
for all $x\in H$. 
The map $\mu_{\alpha}$ possesses the following properties
(see \cite[Sect.~4]{AK}).

\begin{prp}\label{lem-mu-univ}
(a) The map $\mu_{\alpha} : T(X_H) \to  S(t_H) \otimes {}^{\alpha} H$ is an
$H$-comodule algebra map.

(b) For every $H$-comodule algebra map $\mu : T(X_H) \to {}^{\alpha} H$, 
there is a unique algebra map 
$\chi : S(t_H) \to k$
such that 
\[
\mu = (\chi \otimes \id) \circ \mu_{\alpha} \, .
\]
\end{prp}

In other words, any $H$-comodule algebra map $\mu : T(X_H) \to {}^{\alpha} H$ 
can be obtained from~$\mu_{\alpha}$ by specialization.
For this reason we call $\mu_{\alpha}$ the 
\emph{universal comodule algebra map} for~${}^{\alpha} H$.

\begin{theorem}\label{detect}
An element $P \in T(X_H)$ is an $H$-identity for~${}^{\alpha} H$ 
if and only if $\mu_{\alpha}(P) = 0$; equivalently,
\[
I_H({}^{\alpha} H) = \ker(\mu_{\alpha}) \, .
\]
\end{theorem}

This result is a consequence of Proposition~\ref{lem-mu-univ}.
It allows us to detect the $H$-iden\-tit\-ies for any twisted comodule algebra:
it suffices to check them in the easily controllable algebra $S(t_H) \otimes {}^{\alpha}H$.
In \S~\ref{Sweedler} we shall show how to apply this result in an interesting example.

Let us derive some consequences of Theorem~\ref{detect}.
To simplify notation, 
we denote the ideal of $H$-identities $I_H({}^{\alpha} H)$ by~$I_H^{\alpha}$,
the universal $H$-comodule algebra $\UU_H({}^{\alpha} H)$ by~$\UU_H^{\alpha}$,
and the center~$\ZZ_H({}^{\alpha} H)$ of~$\UU_H^{\alpha}$ by~$\ZZ_H^{\alpha}$.

\pagebreak[3]

\begin{crl}\label{injection}
(a) The map $\mu_{\alpha} : T(X_H) \to  S(t_H) \otimes {}^{\alpha} H$
induces an injection of comodule algebras
\[
\overline{\mu}_{\alpha} : \UU_H^{\alpha} \, \hookrightarrow \,  S(t_H) \otimes {}^{\alpha} H \, .
\]

(b) An element of~$\UU_H^{\alpha}$ belongs to the subalgebra~$(\UU_H^{\alpha})^H$
of coinvariants if and only if its image under~$\overline{\mu}_{\alpha}$ 
sits in the subalgebra $S(t_H) \otimes u_1$.
\end{crl}

\pagebreak[3]

We also proved 
that an element of~$\UU_H^{\alpha}$ belongs to the center~$\ZZ_H^{\alpha}$
if and only if its image under~$\overline{\mu}_{\alpha}$ 
sits in the subalgebra 
$S(t_H) \otimes Z({}^{\alpha} H)$, 
where $Z({}^{\alpha} H)$ is the center of~${}^{\alpha} H$
(see~\cite[Prop.~8.2]{AK}).
In particular, since $u_1$ is central in~${}^{\alpha} H$, it follows that all coinvariant elements 
of~$\UU_H^{\alpha}$ belong to the center~$\ZZ_H^{\alpha}$.

We mention another consequence: it asserts that there always exist non-zero $H$-identities
for any non-trivial finite-dimensional twisted comodule algebra.

\pagebreak[3]

\begin{crl}
If $2 \leq \dim_k H < \infty$, then $I_H^{\alpha} \neq \{0\}$.
\end{crl}

\begin{proof}
Suppose that $I_H^{\alpha} = \{0\}$. Then in view of $\UU_H^{\alpha} = T(X_H) / I_H^{\alpha}$
and of Corollary~\ref{injection}, we would have an injective linear map
\[
T^r(X_H) \, \hookrightarrow \, S^r(X_H) \otimes  {}^{\alpha} H 
\]
for all $r\geq 0$. (Here $S^r(X_H)$ is the subspace of elements of degree~$r$ in~$S(t_H)$.) 
Taking dimensions and setting $\dim_k H = n$, we would obtain the inequality
\[
n^r \leq n 
\begin{pmatrix}
r+n-1\\
n-1
\end{pmatrix} \, ,
\]
which is impossible for large~$r$.
\end{proof}

\section{Localizing the universal comodule algebra}\label{sec-local}

We now wish to address the question raised in~\S~\ref{univ-comod-alg}
in the case $A$ is a twisted comodule algebra of the form~${}^{\alpha} H$, where
$H$ is a Hopf algebra and $\alpha$ is a normalized convolution-invertible
two-cocycle on~$H$.

\subsection{The generic base algebra}\label{subsec-generic}

Recall the symmetric algebra~$S(t_H)$ introduced in~\S~\ref{univ-map}.
By~\cite[Lemma~A.1]{AK}
there is a unique linear map $x \mapsto t^{-1}_x$
from~$H$ to the field of fractions~$\Frac S(t_H)$ of~$S(t_H)$
such that for all $x\in H$,
\begin{equation*}\label{tt}
\sum_{(x)}\, t_{x_{(1)}} \, t^{-1}_{x_{(2)}}
= \sum_{(x)}\, t^{-1}_{x_{(1)}} \,  t_{x_{(2)}} = \eps(x) \, 1 \, .
\end{equation*}
(The algebra of fractions generated by the elements $t_x$ and~$t^{-1}_x$ ($x \in H$)
is Takeuchi's free commutative Hopf algebra on the coalgebra underlying~$H$;
see~\cite{Ta}.)

\begin{exas}
(a) If $g$ is a \emph{group-like} element, i.e., $\Delta(g) = g \otimes g$ and $\eps(g) = 1$,
then
\[
t^{-1}_g = \frac{1}{t_g} \, .
\]

(b) If $x$ is a \emph{skew-primitive} element, i.e., $\Delta(x) = g \otimes x + x \otimes h$ for some
group-like elements $g,h$,
then
\[
t^{-1}_x = - \frac{t_x}{t_gt_h} \, .
\]
\end{exas}

For $x,y \in H$, define the following elements of the fraction field~$\Frac S(t_H)$:
\begin{equation*}\label{sigma-def}
\sigma(x,y) = \sum_{(x), (y)}\, t_{x_{(1)}} \, t_{y_{(1)}} \,
\alpha(x_{(2)},y_{(2)}) \, t^{-1}_{x_{(3)} y_{(3)}}
\end{equation*}
and 
\begin{equation*}\label{sigma^{-1}-def}
\sigma^{-1}(x,y) = \sum_{(x), (y)}\, t_{x_{(1)} y_{(1)}} \, 
\alpha^{-1}(x_{(2)},y_{(2)}) \, t^{-1}_{x_{(3)}}  \, t^{-1}_{y_{(3)}}\, ,
\end{equation*}
where $\alpha^{-1}$ is the inverse of~$\alpha$.

The map $(x,y) \in H \times H \mapsto \sigma(x,y) \in \Frac S(t_H)$
is a two-cocycle with values in the fraction field~$\Frac S(t_H)$.

\begin{definition}
The generic base algebra is the subalgebra~$\BB_H^{\alpha}$ of~$\Frac S(t_H)$ 
generated by the elements~$\sigma(x,y)$ and~$\sigma^{-1}(x,y)$,
where $x$ and $y$ run over~$H$.
\end{definition}

Since $\BB_H^{\alpha}$ is a subalgebra of the field~$\Frac S(t_H)$, 
it is a domain and the Krull dimension of~$\BB_H^{\alpha}$
cannot exceed the Krull dimension of~$S(t_H)$, which is~$\dim_k H$.
Actually, it is proved in~\cite[Cor.~3.7]{KM} that if the Hopf algebra~$H$ is
finite-dimensional, then the Krull dimension of~$\BB_H^{\alpha}$ is exactly
equal to~$\dim_k H$.
More properties of the generic base algebra are given in~\cite{KM}.

\begin{exa}
If $H = k[G]$ is the Hopf algebra of a group~$G$ and $\alpha = \alpha_0$
is the trivial two-cocycle, then the generic base algebra $\BB_H^{\alpha}$
is the algebra generated by the Laurent polynomials 
\[
\left ( \frac{t_gt_h}{t_{gh}} \right)^{\pm 1}  ,
\] 
where $g,h$ run over~$G$.
A complete computation for the (in)finite cyclic groups $G = \ZZZ$ and $G = \ZZZ/N$
was given in~\cite[Sect.~3.3]{Ka1}.
\end{exa}

\subsection{Non-degenerate cocycles}

We now restrict to the case when $\alpha$ is a \emph{non-degenerate} two-cocycle,
i.e., when the center of the twisted algebra~${}^{\alpha} H$
is one-dimensional. In this case,
the center of~${}^{\alpha} H$ coincides with the subalgebra of coinvariants.

Recall the injective algebra map
$\overline{\mu}_{\alpha} : \UU_H^{\alpha} \to  S(t_H) \otimes {}^{\alpha} H$
of Corollary~\ref{injection}.
By this corollary and the subsequent comment, 
it follows that in the non-degenerate case
the center~$\ZZ_H^{\alpha}$ of~$\UU_H^{\alpha}$ coincides
with the subalgebra $(\UU_H^{\alpha})^H$ of coinvariants,
and we have
\[
\ZZ_H^{\alpha} = (\UU_H^{\alpha})^H  = \overline{\mu}_{\alpha}^{-1}(S(t_H) \otimes u_1) \, .
\]

The following result connects~$\ZZ_H^{\alpha}$ to
the generic base algebra~$\BB_H^{\alpha}$ introduced in \S~\ref{subsec-generic}
(see~\cite[Prop.~9.1]{AK}). 

\begin{prp}
If $\alpha$ is a non-degenerate two-cocycle on~$H$, then
$\overline{\mu}_{\alpha}$ maps $\ZZ_H^{\alpha}$ into~$\BB_H^{\alpha}\otimes u_1$.
\end{prp}

This result allows us to view the center~$\ZZ_H^{\alpha}$ of~$\UU_H^{\alpha}$
as a subalgebra of the generic base algebra~$\BB_H^{\alpha}$.
It follows from the discussion in \S~\ref{subsec-generic}
that $\ZZ_H^{\alpha}$ is a domain whose Krull dimension is at most~$\dim_k H$.

\pagebreak[3]

We may now consider the $\BB_H^{\alpha}$-algebra
\[
\BB_H^{\alpha} \otimes_{\ZZ_H^{\alpha}} \UU_H^{\alpha} \, .
\]
It is an $H'$-comodule algebra, where $H' = \BB_H^{\alpha} \otimes H$.

The following answers the question raised in~\S~\ref{univ-comod-alg}.

\begin{theorem}\label{answer}
If $H$ is a Hopf algebra and $\alpha$ is a non-degenerate two-cocycle on~$H$
such that $\BB_H^{\alpha}$ is a localization of~$\ZZ_H^{\alpha}$, then
$\BB_H^{\alpha} \otimes_{\ZZ_H^{\alpha}} \UU_H^{\alpha}$ is a cleft $H$-Galois extension
of~$\BB_H^{\alpha}$. 
In particular, there is a comodule isomorphism
\[
\BB_H^{\alpha} \otimes_{\ZZ_H^{\alpha}} \UU_H^{\alpha}
\cong \BB_H^{\alpha} \otimes H \, .
\]
\end{theorem}

It follows that under the hypotheses of the theorem,
a suitable central localization of the universal comodule algebra~$\UU_H^{\alpha}$
is free of rank $\dim_k H$ as a module over its center.

\section{An example: the Sweedler algebra}\label{Sweedler}

We assume in this section that the characteristic of~$k$ is different from~2.

\subsection{Presentation and twisted comodule algebras}

The \emph{Sweedler algebra} $H_4$ is the
algebra generated by two elements $x$, $y$ subject to the relations
\begin{equation*}
x^2 = 1  \, , \quad xy + yx = 0 \, , \quad y^2 = 0 \, .
\end{equation*}
It is four-dimensional. 
As a basis of~$H_4$, we take the set $\{1, x, y, z\}$, where $z = xy$.

The algebra~$H_4$ carries the structure of a 
non-commutative, non-cocom\-mu\-ta\-tive Hopf algebra with
coproduct, counit, and antipode given by
\[
\begin{array}{rclcrcl}
\Delta(1) & = & 1 \otimes 1 \, ,  &\qquad &  \Delta(x) & = &  x \otimes x \, ,\\  
\Delta(y) & = & 1 \otimes y + y \otimes x \, ,  & \;\; &   \Delta(z) & = &  x \otimes z + z \otimes 1 \, ,\\  
\eps(1) & = & \eps(x) =   1 \, ,  & \;\; &   \eps(y) & = &   \eps(z) = 0 \, ,\\  
S(1) & = & 1 \, ,  &\;\; &  S(x) & = &  x  \, ,\\
S(y) & = & z \, ,  &\; \; & S(z) & = &  -y  \, .
\end{array}
\]

The tensor algebra~$T(H_4)$ is the free non-commutative algebra on the four
symbols
\[
E = X_1 \, , \quad X = X_x \, , \quad Y = X_y \, , \quad Z = X_z \, ,
\]
whereas $S(t_{H_4})$ is the polynomial algebra on the symbols $t_1, t_x, t_y, t_z$.

Masuoka~\cite{Ma1} (see also~\cite{DT})
showed that any twisted $H_4$-comodule algebra as in \S~\ref{twisted}
has, up to isomorphism, the following presentation:
\[
{}^{\alpha}H_4
= k \left\langle \, u_x, u_y\, |\, u_x^2 = a u_1\, , \;\; 
u_x u_y + u_y u_x = b u_1 \, , \;\; u_y^2 = c u_1 \, \right\rangle
\]
for some scalars $a$, $b$, $c$ with $a\neq 0$.
To indicate the dependence on the parameters $a, b, c$, we
denote ${}^{\alpha} H_4$ by~$A_{a,b,c}$. 

The center of~$A_{a,b,c}$ consists of the scalar multiples of the unit~$u_1$ 
for all values of $a$, $b$,~$c$. In other words, all two-cocycles on~$H_4$
are non-degenerate.

The coaction $\delta : A_{a,b,c} \to A_{a,b,c} \otimes H_4$ is determined by
\[
\delta(u_x) = u_x \otimes x
\quad\text{and}\quad
\delta(u_y) = u_1 \otimes y + u_y \otimes x \, .
\]
As observed in \S~\ref{twisted}, the coinvariants of~$A_{a,b,c}$ consists of 
the scalar multiples of the unit~$u_1$. 
Therefore, coinvariants and central elements of~$A_{a,b,c}$ coincide.

\subsection{Identities}

In this situation, the universal comodule algebra map 
\[
\mu_{\alpha} : T(X_H) \to S(t_H) \otimes A_{a,b,c}
\]
is given by
\begin{eqnarray*}
\mu_{\alpha}(E) = t_1u_1 \, , && \qquad\mu_{\alpha}(X) = t_x u_x \, ,\\
\mu_{\alpha}(Y) = t_1u_y + t_y u_x \, , && \qquad\mu_{\alpha}(Z) = t_x u_z + t_z u_1\, .
\end{eqnarray*}

Let us set 
\[
R = X^2 \, , \quad  S = Y^2 \, , \quad  T = XY + YX \, , \quad  U = X(XZ + ZX) \, .
\]

\begin{lemma}\label{lem-RSTU}
In the algebra $S(t_H) \otimes A_{a,b,c}$ we have the following equalities:
\begin{eqnarray*}
\mu_{\alpha}(R) & = & at_x^2 \, u_1 \, , \\
\mu_{\alpha}(S) & = & (at_y^2 + b t_1t_y + c t_1^2) \, u_1 \, , \\
\mu_{\alpha}(T) & = & t_x(2at_y + bt_1)\, u_1 \, , \\
\mu_{\alpha}(U) & = & at_x^2(2t_z + bt_x)\, u_1 \, .
\end{eqnarray*}
\end{lemma}

\begin{proof}
This follows from a straightforward computation. 
Let us compute $\mu_{\alpha}(S)$ as an example. We have
\begin{eqnarray*}
\mu_{\alpha}(S) & = & \mu_{\alpha}(Y)^2 
= (t_1u_y + t_y u_x)^2 \\
& = & t_y^2u_x^2 + t_1t_y(u_xu_y + u_y u_x) +  t_1^2 u_y^2 \\
& = & (at_y^2 + b t_1t_y + c t_1^2) \, u_1
\end{eqnarray*}
in view of the definition of~$A_{a,b,c}$.
\end{proof}

We now exhibit two non-trivial $H_4$-identities.

\begin{prp}\label{two-inv}
The elements 
\begin{equation*}
T^2 - 4 RS - \frac{b^2- 4ac}{a} \, E^2 R
\quad\text{and}\quad
ERZ - RXY - \frac{EU - RT}{2}
\end{equation*}
are $H_4$-identities for~$A_{a,b,c}$.
\end{prp}

\begin{proof}
It suffices to check that these two elements are killed by~$\mu_{\alpha}$,
which is easily done using Lemma~\ref{lem-RSTU}. 
\end{proof}

Since $E$, $R$, $S$, $T$, $U$ are sent under~$\mu_{\alpha}$ to~$S(t_H) \otimes u_1$,
their images in~$\UU_H^{\alpha}$ belong to the center~$\ZZ_H^{\alpha}$.
We assert that after inverting the elements~$E$ and~$R$, all relations in~$\ZZ_H^{\alpha}$
are consequences of the leftmost relation in Proposition~\ref{two-inv}.
More precisely, we have the following (see~\cite[Thm.~10.3]{AK}).

\begin{theorem}
There is an isomorphism of algebras
\begin{equation*}
\ZZ_H^{\alpha}[E^{-1}, R^{-1}] \cong k[E,E^{-1},R, R^{-1},S,T,U]/(D_{a,b,c}) \, ,
\end{equation*}
where 
\[
D_{a,b,c} = T^2 - 4 RS - \frac{b^2- 4ac}{a} \, E^2 R \, .
\]
\end{theorem}

To prove this theorem, we first check that the generic base algebra~$\BB_H^{\alpha}$
(whose generators we know) is generated by $E,E^{-1},R,R^{-1}, S,T,U$; 
this implies that $\BB_H^{\alpha}$ is the localization 
\[
\BB_H^{\alpha}= \ZZ_H^{\alpha}[E^{-1}, R^{-1}]
\]
of~$\ZZ_H^{\alpha}$.
In a second step, we establish that 
all relations between the above-listed generators of~$\BB_H^{\alpha}$
follow from the sole relation $D_{a,b,c} = 0$. 

\pagebreak[3]

Let us now turn to the universal comodule algebra~$\UU_H^{\alpha}$.
By Proposition~\ref{two-inv}, we have the following relation in~$\UU_H^{\alpha}$, 
where we keep the same notation for the elements of~$T(X_H)$ 
and their images in~$\UU_H^{\alpha}$:
\begin{equation*}
(ER)Z =  (R)XY + \left(\frac{EU - RT}{2} \right) \quad \text{in}\;\; \UU_H^{\alpha} \, .
\end{equation*}
The elements in parentheses being central, it follows from the previous relation that
if we again invert the central elements~$E$ and~$R$, then $Z$ is a linear combination
of~$1$ and~$XY$ with coefficients in~$\BB_H^{\alpha} = \ZZ_H^{\alpha}[E^{-1}, R^{-1}]$. 
Noting that
\[
YX = - XY + T \in -XY + \ZZ_H^{\alpha}  \subset -XY + \BB_H^{\alpha}\, ,
\]
we easily deduce that after inverting~$E$ and~$R$ any element of~$\UU_H^{\alpha}$ 
is a linear combination of $1, X, Y, XY$ over~$\BB_H^{\alpha}$. 

In~\cite{AK} the following more precise result was established (see \emph{loc.\ cit.}, Thm.~10.7).
It answers positively the question of~\S~\ref{univ-comod-alg}.

\begin{theorem}
The localized algebra~$\UU_H^{\alpha}[E^{-1}, R^{-1}]$ is free of rank~$4$
over its center $\BB_H^{\alpha} = \ZZ_H^{\alpha}[E^{-1}, R^{-1}]$, and
there is an isomorphism of algebras
\begin{equation*}
\UU_H^{\alpha}[E^{-1}, R^{-1}] \cong 
\BB_H^{\alpha} \left \langle \xi, \eta \right \rangle / \left(\xi^2 - R , \, \xi \eta + \eta \xi - T, \, \eta^2 - S \right) \, .
\end{equation*}
\end{theorem}

Note that the algebra~$\BB_H^{\alpha}$ coincides with
the subalgebra of coinvariants of~$\UU_H^{\alpha}[E^{-1}, R^{-1}]$.

\subsection{An open problem}

To complete this survey, we state a problem who will hopefully 
attract the attention of some researchers.

Fix an integer $n\geq 2$ and suppose that the ground field~$k$ contains
a primitive $n$-th root~$q$ of~$1$.  
Consider the Taft algebra~$H_{n^2}$, 
which is the algebra generated by two elements $x$, $y$ subject to the relations
\begin{equation*}
x^n = 1  \, , \quad yx = qxy \, , \quad y^n = 0 \, .
\end{equation*}
This is a Hopf algebra of dimension~$n^2$ with coproduct determined by
\[
\Delta(x) = x \otimes x \quad\text{and}\quad
\Delta(y) = 1 \otimes y  + y \otimes x \, .
\]
The twisted comodule algebras~${}^{\alpha} H_{n^2}$
have been classified in~\cite{DT, Ma1}. (All two-cocycles of~$H_{n^2}$ are non-degenerate.)

Give a presentation by generators and relations of the generic base algebra~$\BB_{H_{n^2}}^{\alpha}$
and show that $\BB_{H_{n^2}}^{\alpha}$ is a localization of~$\ZZ_{H_{n^2}}^{\alpha}$.
(By~\cite[Rem.~2.4\,(c)]{KM} it is enough to consider the case where~$\alpha$ is the trivial cocycle.)

\pagebreak[3]

\section*{Acknowledgements}

I wish to extend my warmest thanks to the organizers 
of the Conference on Quantum Groups and Quantum Topology
held at RIMS, Kyoto University, on April 19--20, 2010, and above all to Professor Akira Masuoka,
for giving me the opportunity to explain my joint work~\cite{AK} with Eli Aljadeff.

This work is part of the project ANR BLAN07-3$_-$183390 
``Groupes quantiques~: techniques galoisiennes et d'int\'egration" funded
by Agence Nationale de la Recherche, France.

\end{document}